\documentclass[twoside,12pt]{article}
\usepackage{amsbsy,latexsym,amsopn,amstext,amsxtra,euscript,amscd}
\usepackage{amssymb,amsfonts,amsmath,amsthm}
\usepackage{graphics,graphicx}
\usepackage{float}
\usepackage{color}

\usepackage{hyperref}

\usepackage{tikz}
\usetikzlibrary{arrows,chains,matrix,positioning,scopes}
\usetikzlibrary{positioning,shadows,backgrounds}  
\tikzset{>=stealth',every on chain/.append style={join},
	every join/.style={->}}
\tikzstyle{labeled}=[execute at begin node=$\scriptstyle,
execute at end node=$]

\newcommand*{\ADR}{University of Sopron,  Institute of Mathematics, Hungary. \textit{nemeth.laszlo@uni-sopron.hu}}

\graphicspath{{Fig/}}

\def\H{${\cal H}$}
\def\T{${\cal T}$}
\def\E{${\cal E}$}

\newtheorem{theorem}{Theorem}

\newtheorem{corollary}{Corollary}
\newtheorem{rem}{Remark}

\newcommand*{\TIT}{Tetrahedron trinomial coefficient transform}
\title{\bf \TIT}

\author{L\'aszl\'o N\'emeth\footnote{\ADR}}
\date{}

\begin{document}

\maketitle \thispagestyle{empty}

\begin{abstract}
We introduce the tetrahedron trinomial coefficient transform which takes a Pascal-like arithmetical triangle to a sequence. We define a Pascal-like infinite tetrahedron \H, and prove that the application of the tetrahedron trinomial transform to one face \T\ of \H\  provides the opposite edge \E\  to \T\  in \H. 
It follows from the construction that the other directions in \H\  parallel to \E\ can be obtained similarly. 
In case of Pascal's triangle the sequence generated by the trinomial transform coincides the binomial transform of the central binomial coefficients. \\[1mm]
The  final  publication  is  available  at \href{http://math.colgate.edu/~integers/}{INTEGERS, Volume 19}. \\[1mm]
%\href{http://math.colgate.edu/~integers/}{INTEGERS: The Electronic Journal of Combinatorial Number Theory, Volume 19}. \\[1mm]

\end{abstract}

% 11B65  Binomial coefficients; factorials; $q$-identities 
% 11B75  Other combinatorial number theory
% 05A10  Factorials, binomial coefficients, combinatorial functions Combinatorics

\section{Introduction}

The binomial transform of the sequence $\{a_n\}_0^\infty\in\mathbb{R}^{\infty}$ is defined by the terms $b_n=\sum_{i=0}^{n}\binom{n}{i} a_i$ of the sequence $\{b_n\}_0^\infty\in\mathbb{R}^{\infty}$.
Several studies  \cite{Barbero,Nemeth_binom,Nemeth_trinom,Pan,Spivey} examine the properties and the generalizations of the binomial transform. Now we introduce a 3-dimensional generalization.  

Let an arithmetical triangle \T\ be given  by  $t_j^{i}\in \mathbb{R}$, where $0\leq j \leq i$ and $i,j\in \mathbb{N}$, and the items $t_j^{i}$ are arranged in rows and columns according to indices $i$ and $j$, respectively.  For example, in case of Pascal's triangle the items $t_j^{i}=\binom{i}{j}$ are the classical binomial coefficients.
Let the sequence $\{b_n\}_{n=0}^\infty$ be the tetrahedron trinomial coefficient transform (in short tetrahedron coefficients transfom) on \T\ defined by 
\begin{equation}\label{def:tetrahedron_coeff_transf}
	b_{n} =\sum_{i=0}^{n} \sum_{j=0}^{i} \binom{n}{j,n-i,i-j} t_{j}^{i},
\end{equation}
where the symbol
\begin{equation}\label{def:tetrahedron_coeff}
	\binom{n}{p,q,r}=\frac{n!}{p!\,q!\,r!}=\binom{n}{p}\binom{n-p}{q}
\end{equation} denotes the tetrahedron trinomial coefficient, $n$, $p$, $q$, $r$ are non-negative integers, and $p+q+r=n$. Thus formula \eqref{def:tetrahedron_coeff_transf} with the binomial coefficients is $$b_n=\sum_{i=0}^{n} \sum_{j=0}^{i} \binom{n}{j}\binom{n-j}{n-i}  t_{j}^{i}.$$

The trinomial coefficients have two different combinatorial meanings in use. One is $\binom{n}{p}_2=\sum_{q=0}^{p}\binom{n}{q}\binom{q}{p-q}$, which is the special case of the multinomial coefficient  (see more in \cite{Andrews,Bel2008}). The other interpretation is  the so-called Pascal's pyramid (more precisely Pascal's tetrahedron). Here we use the second meaning (see \eqref{def:tetrahedron_coeff}) and, referring to its geometric origin, we call it tetrahedron trinomial coefficients (in short tetrahedron coefficients).

We note that the seemingly more natural transform
\begin{equation*}
	b^*_{n} =\sum_{i=0}^{n} \sum_{j=0}^{i} \binom{n}{i,j,n-i-j} t_{j}^{i}
\end{equation*}
leads to the same result because of the symmetry of Pascal's tetrahedron.

In this article, we define an arithmetic structure similar to an infinite tetrahedron, where the elements are arranged in levels, rows and columns. Let the tetrahedron \H\  be defined the following way. A (infinite) face of this Pascal-like tetrahedron is the triangle \T\ and let the other elements be given recursively by the sum of the three items according to Figure~\ref{fig:construction_of_tetrahedron}. The exact definition of $h_{j,k}^i$ is 
\begin{eqnarray*}
	h_{j,0}^i&=&t_j^{i}, \hspace{10.3em}  (0\leq j \leq i),\\
	h_{j,k}^{i}&=&h_{j,k-1}^{i-1}+h_{j,k-1}^{i}+h_{j+1,k-1}^{i},\qquad (1\leq k), 
\end{eqnarray*}
where $0\leq k \leq i$, $0\leq j \leq i-k$ and $i,j,k\in \mathbb{N}$.
The indices $i$, $j$ and $k$ show the position of an item in level $i$, in row $k$ parallel to \T\ and in column $j$.

\begin{figure}[ht!]
	\centering
	\includegraphics{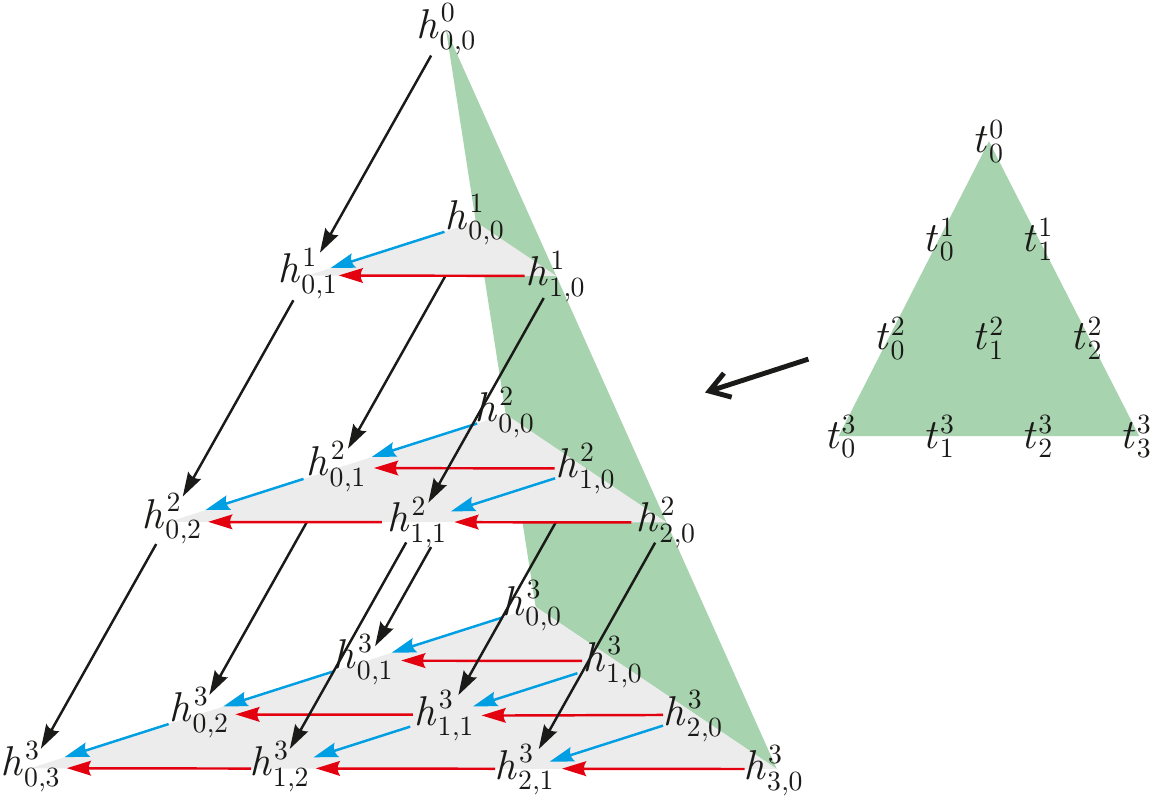}  
	\caption{Construction of the tetrahedron \H}
	\label{fig:construction_of_tetrahedron}
\end{figure}

We shall give some properties of this tetrahedron and show that the edge sequence $\{h_{0,n}^n\}$ of \H\ is the tetrahedron coefficient transform sequence of  \T. Especially, focusing on  Pascal's triangle as \T\ we show, that its tetrahedron coefficient transform is the sequence $b_n=h_{0,n}^n=\sum_{i=0}^{n}\binom{n}{i}\binom{2i}{i}$. Hence  $\{b_n\}$ is the binomial transform of the central binomial coefficients $\binom{2i}{i}$.  Moreover, beside some properties of \H, we gain an identity of binomial coefficients, which is very similar to Vandermonde's convolution formula and a generalization of the well-known identity $\sum_{i=0}^{n} \binom{n}{i}^2  = \binom{2n}{n}$. 
Two similar articles deal with a generalized binomial transform triangle and the trinomial transform triangle in the plane, for more details see \cite{Nemeth_binom,Nemeth_trinom}.

\subsection{Pascal's tetrahedron}

There are several studies dealing with  Pascal's tetrahedron and tetrahedron (trinomial) coefficients (ex.\ \cite{Anatriello,Belb-Szal,Bondarenko,Nemeth_hyppyramid}). In this section,  we give a short summary,  in particular about the properties which will be used later. 

Definition \eqref{def:tetrahedron_coeff} implies immediately the symmetries in variable $p$, $q$ and $r$.

The tetrahedron coefficients also satisfy the following recursive properties. \\
For each $n\geq 0$
$$\binom{n}{n,0,0}=\binom{n}{0,n,0}=\binom{n}{0,0,n}=1.$$
If $n \geq 1$, $pqr\ne 0$, then 
$$\binom{n}{p,q,r}=\binom{n-1}{p-1,q,r}+\binom{n-1}{p,q-1,r}+\binom{n-1}{p,q,r-1},$$
in other cases the so-called Pascal's rule holds, for example, $k=0$ and $pq\ne 0$, $n \geq 1$ give
$$\binom{n}{p,q,0}=\binom{n-1}{p-1,q,0}+\binom{n-1}{p,q-1,0}.$$

Moreover, the tetrahedron coefficients are arranged in Pascal's pyramid (or more precisely  Pascal's tetrahedron), see, e.g., \cite{Anatriello,Bondarenko}. In case of a fixed $n$, we gain the triangular shape level~$n$ of Pascal's tetrahedron. 
Pascal's tetrahedron has been developed by laying the levels below each other. The recursive properties yield that its sides are Pascal's triangles and inside the tetrahedron an item is the sum of the three items directly above it. For some illustrations and its generalization to the 3-dimensional hyperbolic space see \cite{Nemeth_hyppyramid}. Figure~\ref{fig:Pascal_tetrahedron} depicts Pascal's pyramid up to level~4 not in a usual "up-to-down" way, but the "left-to-right" representation which is more suitable for our discussion in the next section.    

\begin{figure}[!ht]
	\centering
	\includegraphics{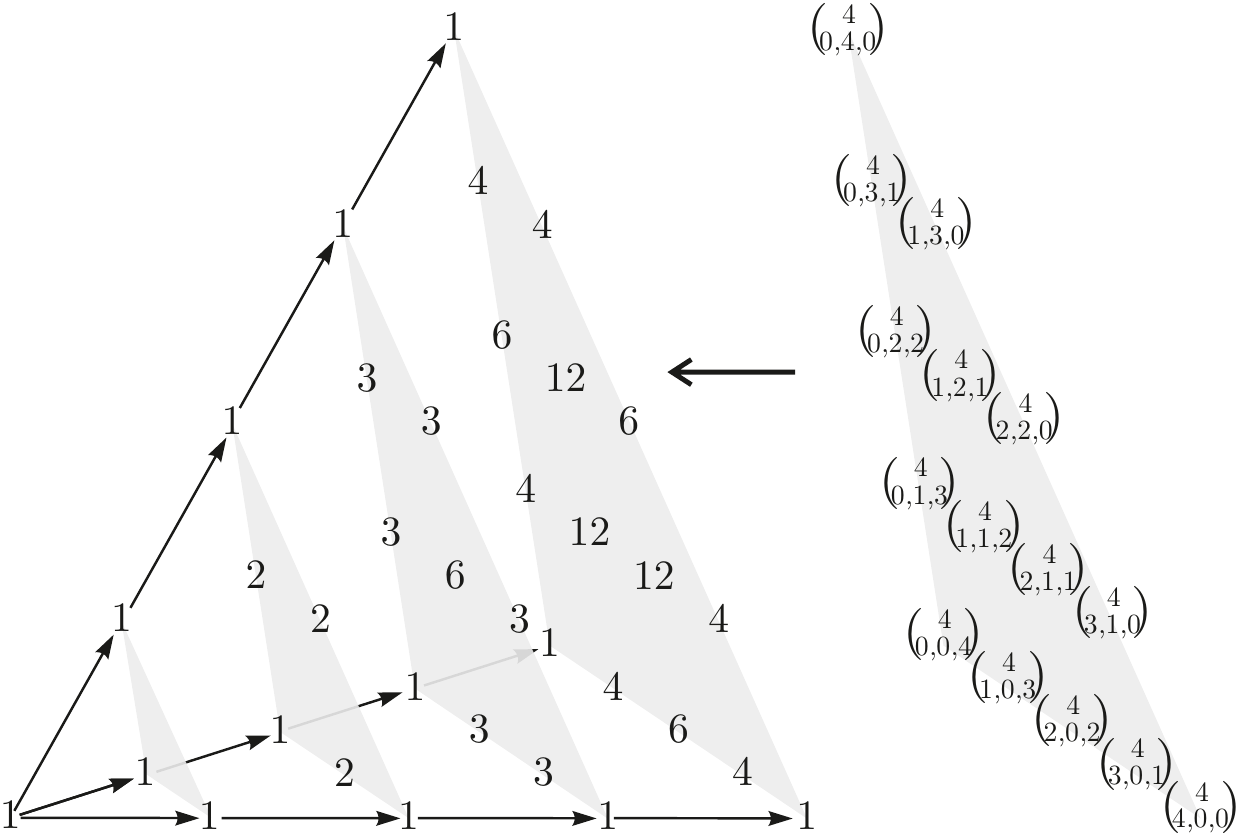}  
	\caption{Pascal's tetrahedron}
	\label{fig:Pascal_tetrahedron}
\end{figure}

\section{Tetrahedron coefficient transform and tetrahedron}

The value of item $h^i_{j,k}$  of \H\ is determined by the value of the terms $t^i_{j}$ of \T\ with tetrahedron coefficients. The next theorem gives the exact formula. For its geometrical background we have to merge the tetrahedron  \H\ and  Pascal's tetrahedron (Figures \ref{fig:construction_of_tetrahedron} and \ref{fig:Pascal_tetrahedron}). 
Putting the start vertex  (the first 1) of the  Pascal's tetrahedron to the position of $h^i_{j,k}$, the value of $h^i_{j,k}$ is the sum of the products of the elements of \T\ and the elements in level~$k$ of Pascal's tetrahedron which are in the same positions. For example, Figure~\ref{fig:tetra_and_Pas} shows the construction of $h_{1,2}^3$ by the help of the 2\textsuperscript{nd} layer of Pascal's tetrahedron and a suitable part of \T.

\begin{figure}[!ht]
	\centering
	\includegraphics{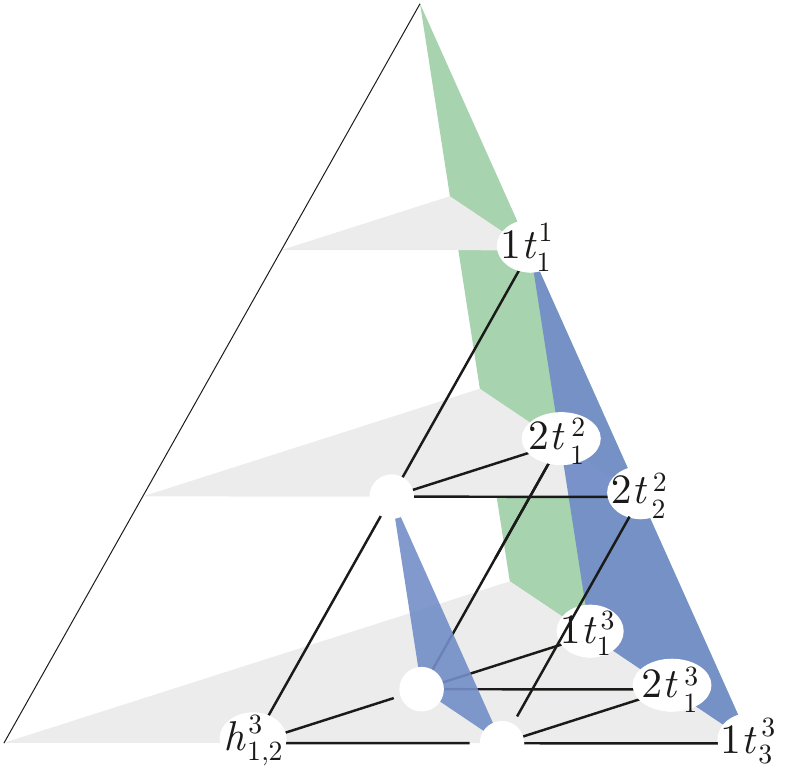}  
	\caption{Value of $h_{1,2}^3$ from Pascal's tetrahedron and triangle \T}
	\label{fig:tetra_and_Pas}
\end{figure}

\begin{theorem} \label{th:tetra_base}
For any $0\leq k \leq i$, $0\leq j \leq i-k$ we have
	\begin{equation*}
	h^i_{j,k}
	=\sum_{r=0}^{k} \sum_{s=0}^{r} \binom{k}{s,k-r,r-s} t_{j+s}^{i-k+r}.
	\end{equation*}
\end{theorem}
\begin{proof} The proof uses induction on $k$. 
	If $k=0$, then the statement trivially holds ($h_{j,0}^i=t_j^i$).
For clarity, in the case $k=1$ we have 	
\begin{eqnarray*}
 h_{j,1}^i&=&h_{j,0}^{i-1}+h_{j,0}^{i}+h_{j+1,0}^{i}
		=\binom{1}{0,1,0}h_{j,0}^{i-1}+\binom{1}{0,0,1}h_{j,0}^{i}+\binom{1}{1,0,0}h_{j+1,0}^{i}\\
		&=&\sum_{r=0}^{1} \sum_{s=0}^{r} \binom{1}{s,1-r,r-s} t_{j+s}^{i+r-1}.
\end{eqnarray*} 
Let us suppose that the theorem is true for up to $k-1$. Thus for any $i$ and $j$ the  induction hypothesis is given by
	$$h^i_{j,k-1}
	=\sum_{r=0}^{k-1} \sum_{s=0}^{r} \binom{k-1}{s,k-r-1,r-s} t_{j+s}^{i-k+r-1}.$$

\noindent Then using the properties of the tetrahedron and binomial coefficients we have 
\begin{eqnarray*}
		h_{j,k}^{i}&=&h_{j,k-1}^{i-1}+h_{j,k-1}^{i}+h_{j+1,k-1}^{i}\\
		&=&
		\sum_{r=0}^{k-1} \sum_{s=0}^{r} \binom{k-1}{s,k-r-1,r-s} t_{j+s}^{i-k+r}
		+ \sum_{r=0}^{k-1} \sum_{s=0}^{r} \binom{k-1}{s,k-r-1,r-s} t_{j+s}^{i-k+r+1}
		\\&& +\sum_{r=0}^{k-1} \sum_{s=0}^{r} \binom{k-1}{s,k-r-1,r-s} t_{j+s+1}^{i-k+r+1} \\
   &=&
		\sum_{r=0}^{k-1} \sum_{s=0}^{r} \binom{k-1}{s,k-r-1,r-s} t_{j+s}^{i-k+r}
		+ \sum_{r=1}^{k} \sum_{s=0}^{r-1} \binom{k-1}{s,k-r,r-s-1} t_{j+s}^{i-k+r}
		\\&& +\sum_{r=1}^{k} \sum_{s=1}^{r} \binom{k-1}{s-1,k-r,r-s} t_{j+s}^{i-k+r}\\ 
		&=&
		\binom{k-1}{0,k-1,0} t_{j}^{i-k} + \sum_{r=1}^{k-1} \sum_{s=0}^{r} \binom{k-1}{s,k-r-1,r-s} t_{j+s}^{i-k+r}
		\\&&+
		\sum_{s=0}^{k-1} \binom{k-1}{s,k-k,k-s-1} t_{j+s}^{i-k+k} +\sum_{r=1}^{k-1} \sum_{s=0}^{r-1} \binom{k-1}{s,k-r,r-s-1} t_{j+s}^{i-k+r}
		\\&&  \sum_{s=1}^{k} \binom{k-1}{s-1,k-k,k-s} t_{j+s}^{i-k+k}+\sum_{r=1}^{k-1} \sum_{s=1}^{r} \binom{k-1}{s-1,k-r,r-s} t_{j+s}^{i-k+r}\\ 
		&=&
		t_{j}^{i-k} + \sum_{s=0}^{k-1} \binom{k-1}{s,0,k-s-1} t_{j+s}^{i}+ \sum_{s=1}^{k} \binom{k-1}{s-1,0,k-s} t_{j+s}^{i}
		\\&& +
		\sum_{r=1}^{k-1} \sum_{s=0}^{r} \binom{k-1}{s,k-r-1,r-s} t_{j+s}^{i-k+r}
		+\sum_{r=1}^{k-1} \sum_{s=0}^{r-1} \binom{k-1}{s,k-r,r-s-1} t_{j+s}^{i-k+r}
		\\&&  +\sum_{r=1}^{k-1} \sum_{s=1}^{r} \binom{k-1}{s-1,k-r,r-s} t_{j+s}^{i-k+r}\\ 
	\end{eqnarray*}	
\begin{eqnarray*}
\phantom{h_{j,k}^{i}}  &=&
		t_{j}^{i-k} + t_{j}^{i}+ t_{j+k}^{i}+ 
		\sum_{s=1}^{k-1} \left( \binom{k-1}{s,0,k-s-1} +  \binom{k-1}{s-1,0,k-s} \right)  t_{j+s}^{i}
		\\&& +
		\sum_{r=1}^{k-1} 
		\left(
		\binom{k-1}{0,k-r-1,r}t_{j}^{i-k+r} + \binom{k-1}{r,k-r-1,0}t_{j+r}^{i-k+r} \right.
		\\&& \qquad \qquad+\sum_{s=1}^{r-1} \binom{k-1}{s,k-r-1,r-s} t_{j+s}^{i-k+r} 
		\\&& \qquad \qquad+
		\binom{k-1}{0,k-r,r-1}t_{j}^{i-k+r} + \sum_{s=1}^{r-1} \binom{k-1}{s,k-r,r-s-1} t_{j+s}^{i-k+r}
		\\&& \qquad \qquad +\binom{k-1}{r-1,k-r,0} t_{j+r}^{i-k+r}+ \left.\sum_{s=1}^{r-1} \binom{k-1}{s-1,k-r,r-s} t_{j+s}^{i-k+r}
		\right)   \\
  &=&
		t_{j}^{i-k} + t_{j}^{i}+ t_{j+k}^{i}+ 
		\sum_{s=1}^{k-1}  \binom{k}{s,0,k-s}   t_{j+s}^{i}
		\\&& +
		\sum_{r=1}^{k-1} 
		\left( 
		\binom{k}{0,k-r,r} t_{j}^{i-k+r}+ \binom{k}{r,k-r,0}  t_{j+r}^{i-k+r} 
		+
		\sum_{s=1}^{r-1} \binom{k}{s,k-r,r-s} t_{j+s}^{i-k+r}
		\right)\\ 
			&=&
		t_{j}^{i-k} +\sum_{s=0}^{k}  \binom{k}{s,0,k-s}   t_{j+s}^{i}
		+
		\sum_{r=1}^{k-1} 
		\sum_{s=0}^{r} \binom{k}{s,k-r,r-s}  
		t_{j+s}^{i-k+r}\\
	&=&
		t_{j}^{i-k}  +
		\sum_{r=1}^{k} 
		\sum_{s=0}^{r} \binom{k}{s,k-r,r-s}  
		t_{j+s}^{i-k+r}
	\\	&=&
		\sum_{r=0}^{k} 
		\sum_{s=0}^{r} \binom{k}{s,k-r,r-s}
		t_{j+s}^{i-k+r}.
	\end{eqnarray*}
\end{proof}

\begin{theorem} The items $h^n_{0,n}$ of \H\ in edge \E\ opposite to \T\ form the tetrahedron coefficient transform sequence $\{b_n\}$ of \T.
\end{theorem}
\begin{proof}
Let $i=k=n$ and $j=0$. Then from Theorem~\ref{th:tetra_base} we gain
	\begin{equation*}
	h^n_{0,n}
	=\sum_{r=0}^{n} \sum_{s=0}^{r} \binom{n}{s,n-r,r-s} t_{s}^{r},
	\end{equation*}
which is the same as \eqref{def:tetrahedron_coeff_transf} if $b_n=h^n_{0,n}$.
\end{proof}

Re-indexing the result of Theorem~\ref{th:tetra_base} we obtain the values of $h^i_{j,k}$ by simple indexed forms of terms of $\mathcal{T}$.
\begin{corollary} For any $0\leq k \leq i$, $0\leq j \leq i-k$ we have
	\begin{equation*}
	h^i_{j,k}
	=\sum_{r=i-k}^{i} \sum_{s=j}^{j+k-i+r} \binom{k}{s-j,i-r,k+j+r-s-i} t_{s}^{r}.
	\end{equation*}
\end{corollary}

In the following, we show that Pascal's rule (an item inside an arithmetical triangle is the sum of items directly above it) and symmetry are inherited in \H.  

\begin{theorem}If $2\leq i, 1\leq j\leq i-1$ and  $t^i_{j}=t^{i-1}_{j-1}+t^{i-1}_{j}$ hold for \T, then for any $0\leq k+2 \leq i$, $1\leq j \leq i-1-k$ we have	
	\begin{equation*}
		h^i_{j,k}=h^{i-1}_{j-1,k}+h^{i-1}_{j,k}.
	\end{equation*}
\end{theorem}	
\begin{proof} The proof is by induction on $k$. The case $k=0$ is trivial. Recall $h^i_{j,0}=t_j^i$. Suppose that the statement of the theorem holds for elements with index $k-1$. Then we have
\begin{eqnarray*}
	h_{j,k}^{i}&=&h_{j,k-1}^{i-1}+h_{j,k-1}^{i}+h_{j+1,k-1}^{i}\\
	&=&h_{j-1,k-1}^{i-2}+h_{j,k-1}^{i-2} +h_{j-1,k-1}^{i-1}+h_{j,k-1}^{i-1} +h_{j,k-1}^{i-1}+h_{j+1,k-1}^{i-1}\\
	&=&h_{j-1,k-1}^{i-2}+h_{j-1,k-1}^{i-1}+h_{j,k-1}^{i-1}+
	h_{j,k-1}^{i-2} +h_{j,k-1}^{i-1} +h_{j+1,k-1}^{i-1}\\
	&=&h^{i-1}_{j-1,k}+h^{i-1}_{j,k}.
\end{eqnarray*}
\end{proof}

\begin{theorem} If \T\ has a vertical symmetry axis, i.e., $t_j^i=t_{i-j}^i$ ($ 0\leq j\leq i$), then \H\ has a symmetry plane, more precisely 
	\begin{equation}\label{eq:h_symmetry}
	h_{j,k}^i=h_{i-(j+k),k}^i,
	\end{equation}
where $0\leq k \leq i$, $0\leq j \leq i-k$.
\end{theorem}
\begin{proof}
With induction on $k$, if $k=0$, then we gain $h_{j,0}^i=t_j^i=t_{i-j}^i=h_{i-(j+0),k}^i$. Suppose that \eqref{eq:h_symmetry} holds in the case $k-1$. Then we have
\begin{eqnarray*}
h_{i-(j+k),k}^i &=& h_{i-(j+k),k-1}^{i-1}+h_{i-(j+k),k-1}^{i}+h_{i-(j+k)+1,k-1}^{i}\\
&=&h_{j,k-1}^{i-1}+h_{j,k-1}^{i}+h_{j+1,k-1}^{i}=h_{j,k}^{i}.
\end{eqnarray*}
\end{proof}

\begin{theorem} \label{th:recursive_T}
Let us suppose that $t_0^i=c$ for any $i\geq0$, where $c$ is a constant and  Pascal's rule holds for \T. Moreover, let $m=j+k+1$. 
Then the sequence  $\{h_{j,k}^{i}\}_{i=j+k}^{\infty}$ is an $m^{\text{th}}$ order homogeneous linear recurrence sequence (in index $i$), given by
\begin{equation} \label{eq:recursive_T}
 \sum_{\ell=0}^{m}(-1)^{\ell}\binom{m}{\ell} h_{j,k}^{i+\ell}=0, 
\end{equation}
where $i\geq j+k$.
\end{theorem}
\begin{proof} The left diagonal is constant in triangle \T, see Figure~\ref{fig:Constan_leftleg}.  The proof is by induction on $k$.
If $k=0$, then $h_{j,0}^{i}=t_j^i$ and 
\begin{equation*} 
\sum_{\ell=0}^{j+1}(-1)^{\ell}\binom{j+1}{\ell}t_{j}^{i+\ell}=0. 
\end{equation*}
We prove it by induction on $j$. 
If $j=0$, the $\sum_{\ell=0}^{1}(-1)^{\ell}\binom{1}{\ell} t_{0}^{i+\ell}=c-c=0$. 
For clarity, see Figure~\ref{fig:Constan_leftleg}. If $j=1$, then $\sum_{\ell=0}^{2}(-1)^{\ell}\binom{2}{\ell} t_{i}^{1+\ell}=t_i^1-2t_1^{i+1}+t_1^{i+2}=((i-1)c+a_1)-2(ic+a_1)+((i+1)c+a_1)=0$. 
Using  the induction hypothesis  $\sum_{\ell=0}^{j}(-1)^{\ell}\binom{j}{\ell} t_{j-1}^{i+\ell}=0$ for any $i\geq j$, then by $t_j^{i+\ell}=\sum_{p=0}^{\ell-1}t_{j-1}^{i+p} + t_{j}^{i}$ (from Pascal's rule) we obtain 
\begin{eqnarray*}
&&\!\!\!\!\!\!\!\!\!\!\!\!\!\!\!\!\!\!\!\!
	\sum_{\ell=0}^{j+1}(-1)^{\ell}\binom{j+1}{\ell}t_{j}^{i+\ell} \ = \  t_{j}^{i} +\sum_{\ell=1}^{j+1}(-1)^{\ell}\binom{j+1}{\ell}\left(t_{j}^{i}+\sum_{p=0}^{\ell-1}t_{j-1}^{i+p} \right)\\
	&=& t_i^j \sum_{\ell=0}^{j+1} (-1)^{\ell}\binom{j+1}{\ell} + \sum_{\ell=1}^{j}(-1)^{\ell}\left(\binom{j}{\ell-1}+\binom{j}{\ell}\right) \sum_{p=0}^{\ell-1}t_{j-1}^{i+p} \\
	& & \qquad + (-1)^{j+1}\binom{j+1}{j+1}  \sum_{p=0}^{j}t_{j-1}^{i+p}\\
	&=&  
	0 + \sum_{\ell=0}^{j-1}(-1)^{\ell+1} \binom{j}{\ell} \sum_{p=0}^{\ell}t_{j-1}^{i+p} + \sum_{\ell=1}^{j}(-1)^{\ell}\binom{j}{\ell} \sum_{p=0}^{\ell-1}t_{j-1}^{i+p}  + (-1)^{j+1} \sum_{p=0}^{j}t_{j-1}^{i+p}\\
	\end{eqnarray*}	
\begin{eqnarray*}
\phantom{h_{j,k}^{i}}	&=&  
(-1) t_{j-1}^{i} + \sum_{\ell=1}^{j-1}(-1)^{\ell+1} \binom{j}{\ell} \left(\sum_{p=0}^{\ell-1}t_{j-1}^{i+p} + t_{j-1}^{i\ell} \right)	
+ \sum_{\ell=1}^{j-1}(-1)^{\ell}\binom{j}{\ell}   \sum_{p=0}^{\ell-1}t_{j-1}^{i+p} \\ & & \qquad+ (-1)^{j}\binom{j}{j}   \sum_{p=0}^{j-1}t_{j-1}^{i+p}  	 
+ (-1)^{j+1}  \sum_{p=0}^{j-1}t_{j-1}^{i+p} + (-1)^{j+1}  t_{j-1}^{i+j}\\
&=& - t_{j-1}^{i} + \sum_{\ell=1}^{j-1}(-1)^{\ell+1} \binom{j}{\ell} t_{j-1}^{i+\ell}  + 0 + 0 + (-1)^{j+1}  t_{j-1}^{i+j}\\
&=& - \sum_{\ell=0}^{j}(-1)^{\ell}\binom{j}{\ell} t_{j-1}^{i+\ell}  =0.
\end{eqnarray*}

Supposing the result is true in the case $k-1$ for any $j$, similarly to the previous proof, we have 
\begin{eqnarray*}
\lefteqn{\sum_{\ell=0}^{j+k+1}(-1)^{\ell}\binom{j+k+1}{\ell}h_{j,k}^{i+\ell}=
 \sum_{\ell=0}^{j+k+1}(-1)^{\ell} \binom{j+k+1}{\ell} \left(h_{j,k-1}^{i-1+\ell}+h_{j,k-1}^{i+\ell}+h_{j+1,k-1}^{i+\ell}\right)} \\
 &=& 
   \left(h_{j,k-1}^{i-1}+h_{j,k-1}^{i}\right) + \left(h_{j,k-1}^{i+j+k}+h_{j,k-1}^{i+j+k+1}\right)
    + 
   \sum_{\ell=1}^{j+k}(-1)^{\ell} \binom{j+k+1}{\ell} \left(h_{j,k-1}^{i-1+\ell}+h_{j,k-1}^{i+\ell}\right) \\ 
  &&\qquad +  
  \sum_{\ell=0}^{(j+1)+k}(-1)^{\ell} \binom{(j+1)+k}{\ell} h_{(j+1),k-1}^{i+\ell}\\\\ 
 &=& 
   h_{j,k-1}^{i+j+k}+h_{j,k-1}^{i+j+k+1} + 
   \sum_{\ell=1}^{j+k}(-1)^{\ell} \binom{j+k}{\ell-1} \left(h_{j,k-1}^{i-1+\ell}+h_{j,k-1}^{i+\ell}\right)\\ 
  &&\qquad +  h_{j,k-1}^{i-1}+h_{j,k-1}^{i}+
   \sum_{\ell=1}^{j+k}(-1)^{\ell} \binom{j+k}{\ell} \left(h_{j,k-1}^{i-1+\ell}+h_{j,k-1}^{i+\ell}\right)+0\\
	\end{eqnarray*}	
\begin{eqnarray*}
\phantom{h_{j,k}^{i}} &=& 
  h_{j,k-1}^{i+j+k}+h_{j,k-1}^{i+j+k+1} + 
  \sum_{\ell=0}^{j+k-1}(-1)^{\ell+1} \binom{j+k}{\ell} \left(h_{j,k-1}^{i+\ell}+h_{j,k-1}^{i+1+\ell}\right)\\ 
  &&\qquad +  
   \sum_{\ell=0}^{j+k}(-1)^{\ell} \binom{j+k}{\ell} \left(h_{j,k-1}^{i-1+\ell}+h_{j,k-1}^{i+\ell}\right)\\
 &=& 
   \sum_{\ell=0}^{j+k}(-1)^{\ell+1} \binom{j+k}{\ell} \left(h_{j,k-1}^{i+\ell}+h_{j,k-1}^{i+1+\ell}\right)+0=0.
\end{eqnarray*}
\end{proof}

\begin{figure}[!ht]
	\centering
	\scalebox{0.99}{ \begin{tikzpicture}[->,xscale=2,yscale=0.8, auto,swap]
		\node(a00) at (0,0)    {$c$};
		
		\node (a01) at (-0.5,-1)   {$c$};
		\node (a02) at (0.5,-1)    {$a_1$};
		
		\node (a01) at (-1,-2)   {$c$};
		\node (a02) at (0,-2)   {$c+a_1$};	
		\node (a01) at (1,-2)   {$a_2$};
		
		\node (a02) at (-1.5,-3)   {$c$};		
		\node (a01) at (-0.5,-3)   {$2c+a_1$};
		\node (a01) at (0.5,-3)   {$c+a_1+a_2$};
		\node (a02) at (1.5,-3)   {$a_3$};	
		
		\end{tikzpicture}}
	\caption{Triangle \T\ with constant $c$ in the left diagonal}
	\label{fig:Constan_leftleg}
\end{figure}
\begin{rem}
	Belbachir and Szalay \cite[Theorem~1]{Belb} gave an explicit form for the items of the so-called Generalized Arithmetical Triangles (GAT). Applying this for $t_j^i$ with conditions of Theorem~\ref{th:recursive_T}   we have 
	$$t_j^i=c \sum_{p=0}^{i-j-1}\binom{i-2-p}{j-1} + \sum_{q=0}^{j-1}\binom{i-2-q}{j-1-q}a_{q+1}.$$
\end{rem}

\begin{corollary}
If  the items in the left and right diagonal are the same constant in  \T\  and Pascal's rule holds for the triangle, then \T\ is symmetric and the sequences  $\{h_{j,k}^{i}\}_{i=j+k}^{\infty}$ and $\{h_{i-(j+k),k}^{i}\}_{i=j+k}^{\infty}$ (where $i\geq j+k$) are the same $(j+k+1)^{\text{th}}$ order homogeneous linear recurrence sequence.
\end{corollary}

\section{Tetrahedron with Pascal's triangle}

In this section, we consider the well-known Pascal's triangle as triangle \T. Now the items of \T\ are the binomial coefficients, so that $t_j^i=\binom{i}{j}$,  ($0\leq j \leq i$). Figure~\ref{fig:tranform_tetra_Pascal} illustrates the tetrahedron digraph which depicts the transformation.

\begin{figure}[h!]
	\centering
	\includegraphics{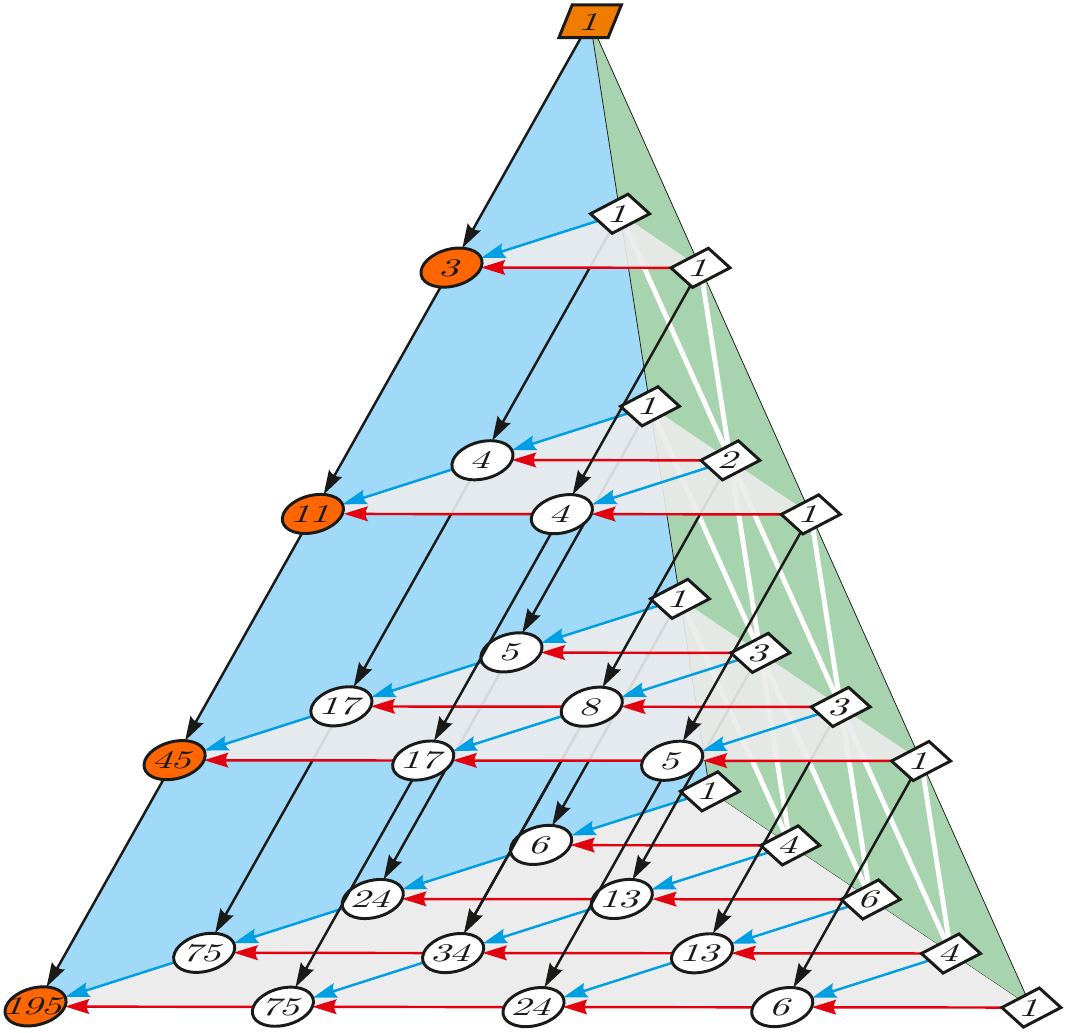}  
	\caption{Tetrahedron coefficient transform based on Pascal's triangle}
	\label{fig:tranform_tetra_Pascal}
\end{figure}

We give an explicit form for the elements of the tetrahedron.
\begin{theorem}\label{th:tetra_transf}
	For any $0\leq k \leq i$, $0\leq j \leq i-k$ we have
	\begin{equation}\label{eq:tetra_transf}
	h_{j,k}^i=\sum_{\ell=0}^{k}\binom{2\ell+i-k}{\ell+j}\binom{k}{\ell}.
	\end{equation}
\end{theorem}
\begin{proof}
	
	Let the triangle \T\ be Pascal's triangle. Then, replacing the tetrahedron coefficients by  the suitable binomial coefficients we have 
\begin{eqnarray*}
	h^i_{j,k}
	&=&\sum_{\ell=0}^{k} \sum_{s=0}^{\ell} \binom{k}{s,k-\ell,\ell-s} t_{j+s}^{i-k+\ell}
	=\sum_{\ell=0}^{k} \sum_{s=0}^{\ell} \binom{k}{k-\ell,s,\ell-s} \binom{i+\ell-k}{j+s}\\
	&=&\sum_{\ell=0}^{k} \sum_{s=0}^{\ell} \binom{k}{k-\ell}\binom{k-(k-\ell)}{s} \binom{i+\ell-k}{j+s}\\
	&=&\sum_{\ell=0}^{k} \binom{k}{\ell} \sum_{s=0}^{\ell} \binom{\ell}{s} \binom{i+\ell-k}{j+s},
\end{eqnarray*}
where $0\leq k \leq i$, $0\leq j \leq i-k$.

On the other hand,  we use  identity \eqref{eq:binom_id}  for the binomial coefficients similar to Vandermonde's convolution formula, which is equation (3.20) in the book of Gould \cite{Gould}, if  $x=n+m$.
For any $0\leq k \leq m$  we have
\begin{equation}\label{eq:binom_id}
\sum_{i=0}^{n} \binom{n}{i} \binom{n+m}{i+k} = \binom{2n+m}{n+k}.
\end{equation}

\noindent Thus $$\sum_{s=0}^{\ell} \binom{\ell}{s} \binom{i+\ell-k}{j+s}=\binom{2\ell+i-k}{\ell+j}$$
proves the theorem.
\end{proof}

In the following, we give a theorem and some sequences generated by \H.

\begin{theorem} The tetrahedron coefficient transform sequence of  Pascal's triangle is the binomial transform sequence of the central binomial coefficients. 
\end{theorem}
\begin{proof}
Let $j=0$ and $i=k=n$. Then from Theorem~\ref{th:tetra_transf} we gain 
\begin{equation*}
   h_{0,n}^n=\sum_{\ell=0}^{n}\binom{2\ell}{\ell}\binom{n}{\ell},
\end{equation*}
where sequence $b_n=h_{0,n}^n$ ($n=0,1,\ldots$) is the binomial transform of $\binom{2\ell}{\ell}$.
\end{proof}

\begin{rem}
	Let $j=j_0$ be fixed and put $k=i-w$, where $0\leq j_0\leq w\leq i$. Then the items of sequence $\{h_{j_0,i-w}^{i}\}_{i=w}^{\infty}$ (in Figure~\ref{fig:tranform_tetra_Pascal} directions with black arrows) are given by  
	\begin{equation*}
	h_{j_0,i-w}^{i}=\sum_{\ell=0}^{i-w}\binom{2\ell+w}{\ell+j_0}\binom{i-w}{\ell}.
	\end{equation*}
	Moreover, because of the symmetry property \eqref{eq:h_symmetry}, $h_{j_0,i-w}^{i}=h_{w-j_0,i-w}^{i}$ ($i=w, w+1, \ldots$) also holds.
\end{rem}

\begin{rem}
	Let $j=j_0$ and $k=k_0$ be fixed. Then sequence $\{h_{j_0,k_0}^{i}\}_{i=j_0+k_0}^{\infty}$ ($0\leq k_0\leq i$, $0\leq j_0\leq i-w-k_0$) is given by \eqref{eq:tetra_transf}. The terms of this sequence are bellow each other in Figure~\ref{fig:tranform_tetra_Pascal} and because of symmetry, we have $h_{j_0,k_0}^{i}=h_{i-(j_0+k_0),k_0}^i$ ($i=j_0+k_0, j_0+k_0+1, \ldots$), where direction of the sequence on the right-hand-side  is parallel to the right diagonal (right leg) of Pascal's triangle. According to Theorem~\ref{th:recursive_T}, these are $(j_0+k_0+1)^{\text{th}}$ order homogeneous linear recurrence sequences with formula \eqref{eq:recursive_T}.
\end{rem}

In the collection of "On-Line Encyclopedia of Integer Sequences" \cite{Sloane}, we found several sequences can be described by our method. Table ~\ref{table:seqences} shows them.
\begin{table}[h!]
	\centering
	\caption{Some sequences in OEIS \cite{Sloane}}
	\label{table:seqences}
	\begin{tabular}{||l|l|l||}
		\hline\hline
	Sequence& OEIS number & First few terms  \\ \hline\hline
		$\{b_n=h_{0,n}^{n}\}$   & A026375          &  $1,3,11,45,195,873,3989,18483,86515,\ldots$             \\ \hline
		$\{h_{0,n-1}^{n}\}$, $\{h_{1,n-1}^{n}\}$& A026378          & $1, 4, 17, 75, 339, 1558, 7247, 34016, 160795,  \ldots$ \\ \hline
		$\{h_{0,n-2}^{n}\}$, $\{h_{2,n-2}^{n}\}$ & A026388          & $1,5,24,114,541,2573,12275,58747,282003,\ldots$        \\ \hline
		$\{h_{0,n-3}^{n}\}$, $\{h_{3,n-3}^{n}\}$ & A034942          & $1, 6, 32, 163, 813, 4013, 19703, 96477, 471811, \ldots$        \\ \hline		
		$\{h_{1,n-2}^{n}\}$ & A085362, A026387 & $2,8,34,150,678,3116,14494,68032,321590,\ldots$        \\ \hline
		$\{h_{1,1}^{n}\}$, $\{h_{n-2,1}^{n}\}$ & A034856 & $4, 8, 13, 19, 26, 34, 43, 53, 64, 76, 89, 103, 118 ,\ldots$        \\ \hline
		$\{h_{2,1}^{n}\}$, $\{h_{n-3,1}^{n}\}$ & A008778 & $5, 13, 26, 45, 71, 105, 148, 201, 265, 341, 430,\ldots$        \\ \hline
		$\{h_{1,2}^{n}\}$, $\{h_{n-3,2}^{n}\}$ & A023545  & $17,34,58,90,131,182,244,318,405,506,622,\ldots$        \\ \hline\hline						
	\end{tabular}
\end{table}

\section{Further work}

Our tetrahedron construction has some connecting points to other known arithmetical triangles, its examination would be profitable. Now, we show some relations. Let us extend the tetrahedron by using the convention $\binom{i}{j}=0$ for $j\notin\{0,1,\ldots,i\}$. (Really, there are more extensions of Pascal's triangle.) This way we obtain the 3-dimensional arithmetical construction in Figure~\ref{fig:extend_tranform_tetra_Pascal}.

The triangle section $\{h_{j,1}^{i}\}$ parallel to Pascal's triangle is the so-called 3-Pascal triangle (A028262). 
The triangle (light-blue triangle in Figure~\ref{fig:extend_tranform_tetra_Pascal}) described by the elements $h_{j,i}^{i}$, where $i\in \mathbb{N}_0, -i\leq j\leq i$ is the "Gegenbauer functional" triangle $T(i,k)$ = $\text{GegenbauerC}(m,-i,-\frac{3}{2})$, where $m = k$ if $k<i$ else $2i-k$, for  $0\leq i$ and $0\leq k\leq 2i$ (see A272866). This conjecture is based on our calculation up to $i=20$.

More sequences and triangles can be found in OEIS, for example, A026376, A026374.

\begin{figure}[!ht]
	\centering
	\includegraphics[width=12cm,keepaspectratio]{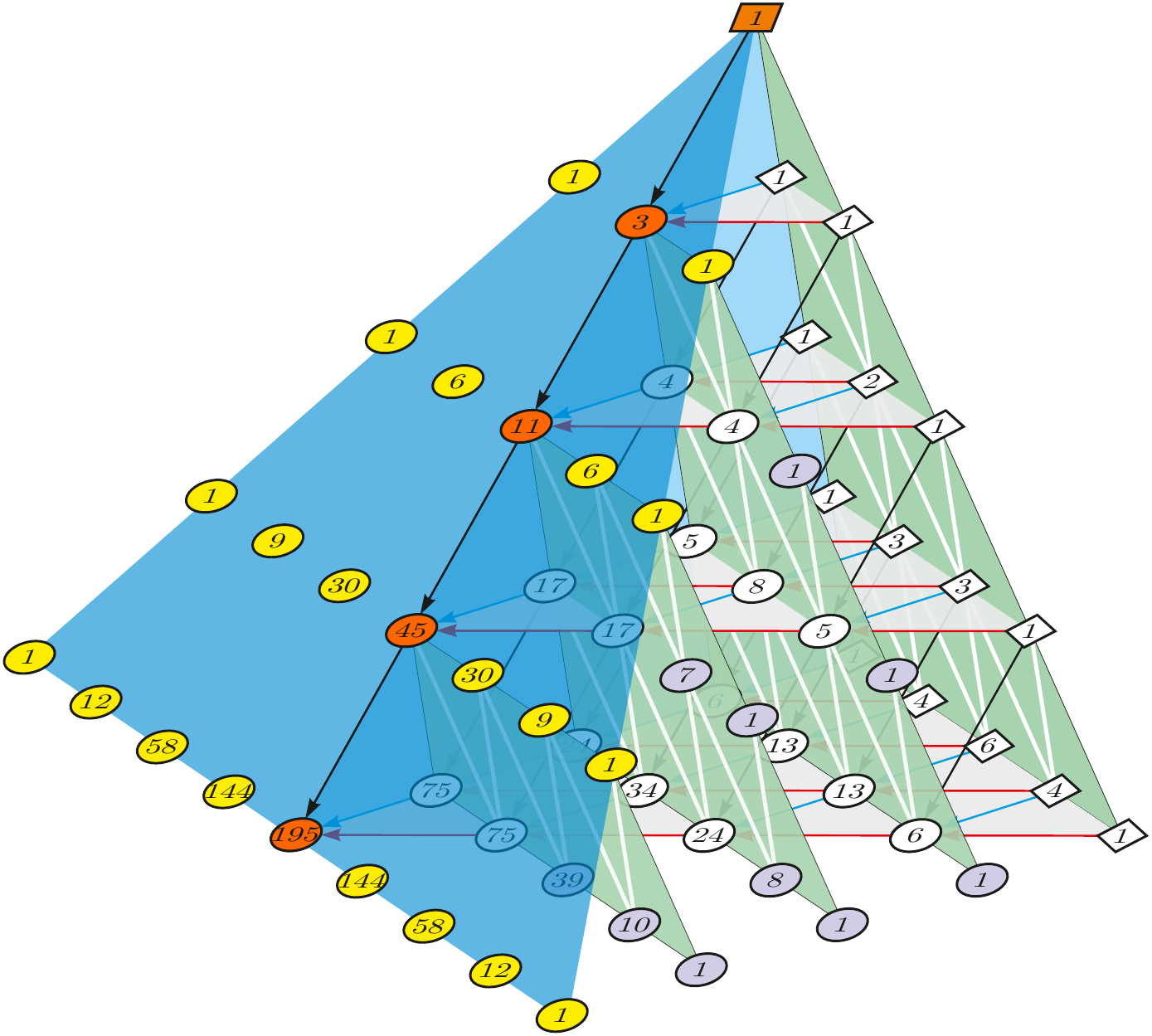}
	\caption{Extended tetrahedron coefficient construction based on Pascal's triangle}
	\label{fig:extend_tranform_tetra_Pascal}
\end{figure}

\section{Acknowledgement}

The author would like to thank the anonymous referee for carefully reading the manuscript and for his/her useful suggestions and improvements.

 \end{document}